\documentclass[11pt]{amsart}

\usepackage{etex}
\usepackage{amsmath, amssymb}
\usepackage{array}
\usepackage[frame,cmtip,arrow,matrix,line,graph,curve]{xy}
\usepackage{graphpap, color, paralist, pstricks}
\usepackage[mathscr]{eucal}
\usepackage[pdftex]{graphicx}
\usepackage[pdftex,colorlinks,backref=page,citecolor=blue]{hyperref}
\usepackage{tikz}
\usepackage{mathtools}

\usepackage{soul,xcolor}

\newtheorem{theorem}{Theorem}[section]
\newtheorem{proposition}[theorem]{Proposition}
\newtheorem{corollary}[theorem]{Corollary}
\newtheorem{lemma}[theorem]{Lemma}

\newtheorem{conjecture}[theorem]{Conjecture}

\theoremstyle{definition}
\newtheorem{definition}[theorem]{Definition}
\newtheorem{example}[theorem]{Example}

\newtheorem{remark}[theorem]{Remark}

\newtheorem{notation}[theorem]{Notation}

\numberwithin{equation}{section}

\newcommand{\FF}{\mathbb{F}}
\newcommand{\PP}{\mathbb{P}}
\newcommand{\QQ}{\mathbb{Q}}
\newcommand{\CC}{\mathbb{C}}

\newcommand{\ZZ}{\mathbb{Z}}

\newcommand{\HH}{\mathbb{H}}

\newcommand{\cO}{\mathcal{O} }

\newcommand{\cE}{\mathcal{E} }

\newcommand{\cM}{\mathcal{M} }

\newcommand{\cV}{\mathcal{V} }
\newcommand{\cU}{\mathcal{U} }

\newcommand{\rE}{\mathrm{E} }
\newcommand{\rH}{\mathrm{H} }
\newcommand{\rM}{\mathrm{M} }

\newcommand{\rP}{\mathrm{P} }

\newcommand{\bH}{\mathbf{H} }
\newcommand{\bK}{\mathbf{K} }

\newcommand{\bM}{\mathbf{M} }

\newcommand{\bX}{\mathbf{X} }
\newcommand{\bY}{\mathbf{Y} }

\def\Sym{\mathrm{Sym} }
\def\Hom{\mathrm{Hom} }
\def\Ext{\mathrm{Ext} }

\def\Gr{\mathrm{Gr} }

\def\SL{\mathrm{SL}}

\def\git{/\!/ }

\def\lr{\rightarrow}
\def\rIH{\mathrm{IH}}
\def\rIE{\mathrm{IE}}
\def\rIP{\mathrm{IP}}

\newcommand{\ses}[3]{0\rightarrow{#1}\rightarrow{#2}\rightarrow{#3}\rightarrow0}

\makeatletter
\providecommand{\leftsquigarrow}{%
  \mathrel{\mathpalette\reflect@squig\relax}%
}
\newcommand{\reflect@squig}[2]{%
 ~\reflectbox{$\m@th#1\rightsquigarrow$}%
}
\makeatother

\begin{document}

\title[Intersection cohomology of pure sheaf space]{Intersection cohomology of pure sheaf spaces using Kirwan's desingularization}

\author{Kiryong Chung}
\address{Department of Mathematics Education, Kyungpook National University, 80 Daehakro, Bukgu, Daegu 41566, Korea}
\email{krchung@knu.ac.kr}

\author{Youngho Yoon}
\address{Applied algebra and Optimization Research Center (AORC), Sungkyunkwan University, 2066 Seobu-ro, Suwon 16419, Korea}
\email{mathyyoon@skku.edu}


\renewcommand{\j}[1]{\textcolor[rgb]{0.98,0.5,0}{\sout{#1}}}
\newcommand{\jadd}[1]{\textcolor[rgb]{0.98,0.5,0}{#1}}

\setstcolor{red}

\begin{abstract}
Let $\mathbf{M}_n$ be the Simpson compactification of twisted ideal sheaves $\mathcal{I}_{L,Q}(1)$ where $Q$ is a rank $4$ quardric hypersurface in $\mathbb{P}^n$ and $L$ is a linear subspace of dimension $n-2$. 
This paper calculates the intersection Poincar\'e polynomial of $\mathbf{M}_n$ using Kirwan's desingularization method. We obtain the intersection Poincar\'e polynomial of the moduli space for one-dimensional sheaves on del Pezzo surfaces of degree $\geq 8$ by considering wall-crossings of stable pairs and complexes.
\end{abstract}
\maketitle


\section{Introduction}\label{sec:intro}

\subsection{Kronecker quiver and related works}

Let $\bM_n$ be the space parameterizing semi-stable sheaves $F$ on the projective space $\PP^n$ with a linear free resolution 
\begin{equation}\label{eqn:resolution}
        0 \to \cO_{\PP^{n}}(-1)\oplus \cO_{\PP^{n}}(-1) {\to} \cO_{\PP^{n}}\oplus \cO_{\PP^{n}} \to F
        \to 0.
\end{equation}
By a deformation theoretic argument of sheaves, $\bM_n$ is an irreducible normal variety with dimension $4n-3$ and generic moduli points of $\bM_n$ parameterize twisted ideal sheaves $\mathcal{I}_{L,Q}(1)$, where $Q$ is a rank $4$ quadric hypersurface in $\mathbb{P}^n$ and $L$ is a linear subspace of dimension $n-2$. The space $\bM_n$ (as a quiver representation space) has been studied  in several areas, including homological mirror symmetry, birational geometry, and curve counting theory.
King (\cite{Kin94}) showed that a general quiver representation space is projective under suitable conditions, and this space has been used in several areas of algebraic geometry. For our purpose in this paper, let $\bK_{n}(a,b)$ be the moduli space of Kronecker quiver representations with dimension vector $(a,b)$ and $(n+1)$-simple arrows (see Section~\ref{quiver} for the complete definition). Hosono and Takagi (\cite{HT16}) used Kronecker modules space $\bK_4(2,2)$ (called the \emph{double symmetroid}) as the starting point to find a pair of derived equivalent but not birationally equivalent Calabi-Yau threefolds. 

In terms of birational geometry, Kontsevich's moduli space $\cM_0(\Gr(d,n+1),d)$ of degree $d$ and genus zero stable maps to $\Gr(d,n+1)$ is birationally equivalent to $\bK_n(2,d)$ (\cite[Proposition 4.8]{CM17}). This paper focuses on cases for $d=2$, since $\bM_n\cong \bK_n(2,2)$ is a minimal birational contraction of $\cM_0(\Gr(2,n+1),2)$ (\cite{CM17}).

On the other hand, the moduli space $\bM_S(c, \chi)$ of semi-stable pure sheaves  $F $ with $c_1(F)=c$ and $\chi(F)=\chi$ on  a del Pezzo surface $S$ has been studied in the virtual curve counting theory, where the Gopakumar-Vafa (GV) invariant of the local surface (i.e., the total space $\mathrm{Tot}(K_S)$ of the canonical line bundle $K_S$ on $S$) is conjectured to be the topological Euler number of the moduli space $\bM_S(c, \chi)$ whenever it is smooth (\cite{Kat08}). One method to compute the Euler number is to use Bridgeland wall-crossings of $\bM_S(c, \chi)$ (\cite{BMW14, CHW14, CC15}). The moduli space is regarded as the moduli space of semi-stable objects on the derived category of $\mathrm{Coh}(S)$. The Kronecker modules space $\bK_n(a,b)$ (or a projective bundle over $\bK_n(a,b)$) naturally arises as the final model of Bridgeland wall-crossings of $\bM_S(c, \chi)$.

This paper calculates the intersection cohomology (of middle perversity) of $\bM_n$ using the geometric invariant theoretic (GIT) quotient description for $\bM_n\cong \bK_n(2,2)$ combined with Kirwan's method (\cite{Kir86b, Kir86a}). Subsequently, we use the result to compute the intersection cohomology group for $\bM_S(c,\chi)$ on del Pezzo surfaces $S$ of degree $\geq8$.

\subsection{Main result and application}

The main result of this paper is the following.
\begin{theorem}\label{mainthm1}
For each integer $n\geq 2$, the intersection Poincar\'e polynomial of $\bM_n$ is
\[
\rIP(\bM_n)=\frac{(1-t^{4n+4})(1-t^{4\lfloor\frac{n}{2}\rfloor})(1-t^{4\lfloor\frac{n+1}{2}\rfloor)}}{(1-t^2)(1-t^4)^2} ,
\]
where $\lfloor x\rfloor$ is the largest integer $\leq x$.
\end{theorem}
The key ingredient of the proof of Theorem~\ref{mainthm1} is that a partial desingularization of $\mathbf{M}_n$ is isomorphic to the moduli space $\cM_0(\Gr(2,n+1),2)$ of degree two stable maps to Grassmannian variety $\Gr(2,n+1)$ (\cite[Theorem 5.1]{CM17}). 
We calculate the intersection Poincar\'e polynomial of $\bM_n$ (\cite{Kir86b, Kir86a}) considering the variation of intersection Betti numbers of intermediate moduli spaces.  One key issue is to check that each term is pure and balanced Hodge type (see (\ref{pureandbalanced}) of Remark~\ref{useful} for the definition). Thus Theorem~\ref{mainthm1} is recovered in the level of the intersection E-polynomial by letting $t^2:=uv$. 

As corollaries of Theorem~\ref{mainthm1}, relating $\bM_n$ and $\bM_S(\beta, \chi)$ using wall-crossings of pairs and complexes, we obtain topological invariants of moduli space $\bM_S(\beta,\chi)$ on del Pezzo surfaces such as the Hirzebruch surface $\FF_k=\PP(\cO_{\PP^1}\oplus\cO_{\PP^1}(-k))$ for $k=0,1$ and the projective plane $\PP^2$. More precisely,
\begin{corollary}\label{cormain}
Let $\bM_S(\beta,\chi)$ be the moduli space of semi-stable sheaves $F$ on a del Pezzo surface $S$ with $c_1(F)=\beta\in \rH_2(S,\ZZ)$ and $\chi(F)=2$. Then the intersection Euler numbers of the moduli space can be expressed as 

\begin{center}
\begin{tabular}{|l|l||p{6.2cm}|}
\hline
$S$&$\beta\in \rH_2(S,\ZZ)$ &Intersection Euler number of $\bM_S(\beta,\chi)$\\
\hline
\hline
$\FF_0$&$c_1(\cO_{\FF_0}(2,2))$&$\qquad\qquad\qquad\;36$\\
\hline
$\FF_1$&$c_1(\cO_{\FF_1}(4,2))$&$\qquad\qquad\qquad\;110$\\
\hline
$\PP^2$&$c_1(\cO_{\PP^2}(4))$&$\qquad\qquad\qquad\;192$\\
\hline
\end{tabular}
\end{center}
\end{corollary}

More generally, we calculate virtual intersection Poincar\'e polynomials of moduli spaces in Corollary~\ref{cor1},~\ref{cor2}, and~\ref{cor3}. A key idea for the proof of these Corollaries is the following. The difference between the E-polynomials and the intersection E-polynomials of a quasi-projective variety $Y$ is completely measured by geometric information from the analytic neighborhood of the singular locus of $Y$ (Corollary~\ref{comparisoneandie}) (\cite{MR2376848}). Since the space $\bM_n$ has the same singularity type as the moduli space $\bM_S(c,\chi)$ (cf. Remark~\ref{luna}), we can obtain the intersection Poincar\'e polynomial of $\bM_S(c,\chi)$ from the singularity type of $\bM_n$.

\begin{remark}
Upon replacing $\chi(F)=2$ by $\chi(F)=1$ in Corollary~\ref{cormain}, the (intersection) Euler numbers of the moduli spaces (GV-numbers) do not change (\cite[Proposition 12]{BH14}, \cite[Proposition 4.9]{CGKT18}, and \cite[Corollary 5.2]{CC17}).
This is interesting, since the moduli spaces may not be isomorphic to each other for different $\chi$. For example, if $S=\PP^2$, $\bM_{\PP^2}(d,\chi)$ is isomorphic to $\bM_{\PP^2}(d',\chi')$ if and only if $d=d'$ and $\chi\equiv\pm\chi'$ (mod $d$) (\cite[Theorem 8.1]{Woo13}). 
\end{remark}
\begin{remark}
The fact that the moduli space of sheaves on the del Pezzo surface $\geq 8$ is birationally equivalent to the space $\bM_n$ gives us a chance to compute the E-polynomial of the moduli spaces.
\end{remark}
\subsection{Structure of this paper}
Section~\ref{sec:preliminaries} reviews geometric properties of several moduli spaces that are subsequently used to calculate the intersection Poincar\'e polynomial of $\bM_n$. We also recall some basic notions and properties related to the (intersection) E-polynomial of a quasi-projective variety. In section~\ref{sec:mainresults} we prove Theorem~\ref{mainthm1} using Kirwan's method (\cite[(2,1) and (2,28)]{Kir86a}) and obtain a numerical relationship with open cones from the singular loci of $\bM_n$ (Corollary~\ref{univrelation}). In section~\ref{ihdoflocalsurface} we calculates the (virtual) intersection Poincar\'e polynomial for moduli space $\bM_S(\beta,\chi)$ on del Pezzo surfaces (Corollaries~\ref{cor1},~\ref{cor2}, and~\ref{cor3}) using explicit birational morphisms and wall-crossings among related spaces.

\subsection*{Acknowledgement}

Some parts of this work were completed while K. Chung attended the Mixed Hodge Modules and Birational Geometry summer course at Johannes Gutenberg University Mainz in July, 2018, and he thanks the organizers for their invitation and hospitality. The authors gratefully acknowledge many helpful suggestions from Seung-Jo Jung, Joonyeong Won, and Sang-Bum Yoo during the preparation of the paper. We also thank the anonymous reviewer for valuable comments and suggestions to improve the quality of the paper.
\section{Preliminaries}\label{sec:preliminaries}

This section reviews several properties for moduli spaces of our interest and the E-polynomial of a quasi-projective variety which we will use.

\subsection{Moduli of stable maps to Grassmannian} 

Let us recall the definition and geometric properties for spaces of stable maps. Let $X$ be a projective variety with a fixed embedding in $\PP^n$, and $C$ be a projective connected reduced curve. A map $f:C \to X$ is called stable if $C$ has at worst nodal singularities and $|\mathrm{Aut}(f)| < \infty$.
Let $\cM_g(X, d)$ be the moduli space of stable maps with arithmetic genus $g(C) = g$ and degree $\mathrm{deg}(f)=d$. If $X$ is a convex variety and $g=0$, then moduli space $\cM_0(X, d)$ is a projective variety with at most finite group quotient singularity (\cite[Theorem 2]{FP97}). 
This paper focuses on the case $X=\Gr(2,n+1)$ and $d=2$, denoted as
\[\bK_n := \cM_0(\Gr(2, n+1), 2).\]
\subsection{Moduli space of semi-stable sheaves}\label{ssec:Simpson}

Let $X$ be a smooth projective variety with a fixed polarization $L$. For a coherent sheaf $F$ on $X$, the Hilbert polynomial $P(F)(m)$ is defined as $\chi(F\otimes L^{m})$. If the support of $F$ has dimension $d$, $P(F)(m)$ has degree $d$ and can be expressed as
\[
        P(F)(m) = \sum_{i=0}^{d}a_{i}\frac{m^{i}}{i!},
\]
where $r(F) := a_{d}$ is called the multiplicity of $F$, and the \emph{reduced Hilbert polynomial} is $p(F)(m) := P(F)(m)/r(F)$. A pure sheaf $F$ is \emph{semi-stable} if for every nonzero proper subsheaf $$F' \subset F,\; p(F')(m) \le p(F)(m)$$ for $m \gg 0$. We say $F$ is \emph{stable} if the inequality is strict. For each semi-stable sheaf $F$, there is a filtration (Jordan-H\"older filtration) $0 = F_{0} \subset F_{1} \subset \cdots \subset F_{n} = F$ such that $\mathrm{gr}_{i}(F) := F_{i}/F_{i-1}$ is stable and $p(F)(m) = p(\mathrm{gr}_{i}(F))(m)$ for all $i$. Finally, two semi-stable sheaves $F$ and $G$ are $S$-\emph{equivalent} if $\mathrm{gr}(F) \cong \mathrm{gr}(G)$, where $\mathrm{gr}(F) := \oplus_{i}\mathrm{gr}_{i}(F)$. Simpson~\cite{Sim94} proved there is a projective coarse moduli space $\bM_{L}(X, P(m))$ of $S$-equivalent classes of semi-stable sheaves for a fixed Hilbert polynomial $P(m)$. This moduli space has several connected components with respect to $\beta=c_1(F) \in H_{2}(X, \ZZ)$, that is,
\[
        \bM_{L}(X, P(m)) = \bigsqcup_{\beta \in H_{2}(X, \ZZ)}
        \bM_{L}(X, \beta, P(m)).
\]
For brevity, we denote $\bM_{L}(X, P(m))$ by $\bM_{X}(P(m))$ and $\bM_{L}(X,\beta, P(m))$ by $\bM_{X}(\beta, P(m))$.

Let $\rM_{\PP^n}(P(m))$ be the moduli space of semi-stable sheaves $F$ with Hilbert polynomial $P(m)=2\chi(\cO_{\PP^{n-1}}(m))-2\chi(\cO_{\PP^{n-1}}(m-1))$. Let $\bM^\circ$ be the space of twisted ideal sheaves $\mathcal{I}_{L,Q}(1)$ such that $L\in |\cO_{\PP^n}(1)|$ and $Q\in |\cO_{\PP^n}(2)|$ with $\mathrm{rank}(Q)=4$. Then we summarize the results of Section 4.2 in \cite{CM17} as follows.
\begin{enumerate}
\item The closure $\bM_n$ of $\bM^\circ$ in $\rM_{\PP^n}(P(m))$ is an irreducible normal variety of dimension $4n-3$.
\item $\bM_n$ is a connected component of $\rM_{\PP^n}(P(m))$.
\item Each semi-stable sheaf $F$ parameterized by $\bM_n$ has a free resolution as in \eqref{eqn:resolution}.
\item The singular locus of $\bM_n$ is isomorphic to $\mathrm{Sing}(\bM_n)\cong \mathrm{Sym}^2(\PP^{n*})$ parameterizing pure sheaves of the form $\cO_{H}\oplus \cO_{H'}$ for hyperplanes $H$ and $H'$ in $\PP^n$.
\end{enumerate}

\begin{remark}\label{luna}
From the resolution \eqref{eqn:resolution} of $F$, $\Ext^2(F,F)=0$, which implies that the Quot scheme arising in the GIT construction of $\bM_n$ (\cite{HuLe10}) is smooth. Thus, from Luna's \'etale slice theorem, the analytic normal neighborhood of $\mathrm{Sing}(\bM_n)$ in $\bM_n$ is the same as that of the moduli space of vector bundles over a smooth projective curve (\cite{Las96} and \cite{Kir86b}). 
\end{remark}

\subsection{Resolution of $\bM_n$ using Kirwan's method}\label{quiver}

The birational relationship between $\bM_n$ and Kontsevich's map space $\bK_n$ was explicitly studied in \cite[Section 5]{CM17} using Kirwan's desingularization method. One critical point is the interpretation of $\bM_n$ as a Kronecker quiver representation space. For convenience, we recall the results in detail.

Fix two positive integers $a, b$ and let $V^*$ be a vector space of $\mathrm{dim}V^*=n+1$. A \emph{Kronecker} $V^*$-\emph{module} is a quiver representation of an $n$-Kronecker quiver 
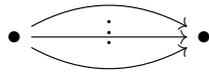
\[
        \xymatrix{\bullet\ar@/^1.0pc/[rr] \ar@/_1.0pc/[rr]\ar[rr]
        &{\vdots}&\bullet\\}
\]
with dimension vector $(a, b)$. Two Kronecker $V^*$-modules $\phi = (\phi_{i})$ and $\psi = (\psi_{i})$ are \emph{equivalent} if there are $A \in \SL_{a}$ and $B \in \SL_{b}$, such that $\phi = B \circ \psi \circ A$. We may regard the GIT quotient 
\[
        \bK_n(a,b):=\PP\Hom(V^*\otimes \CC^{a}, \CC^{b})\git \SL_{a} \times \SL_{b}
\]
as the moduli space of \emph{semi-stable} Kronecker $V^*$-modules. 
We are interested in the case $a=b=2$ and $G:=\SL_{2} \times \SL_{2}$.

\begin{lemma}\label{lem:stability}
Let $M \in \PP\Hom(V^*\otimes \CC^{2}, \CC^{2})\cong \PP(V^{*}\otimes \mathfrak{gl}_{2}):=\bX$. If $M\in \bX^{ss}\setminus \bX^{s}$, then $M$ is equivalent to 
\[
        \left[\begin{array}{cc}g&0\\0&h\end{array}\right]
\]
for some $g, h \in V^{*}\setminus \{0\}$ where
\begin{enumerate}
\item $\mathrm{Stab}\; M \cong \SL_{2} \ltimes \ZZ_{2}$ if $g$ is proportional to $h$ and
\item $\mathrm{Stab}\;M \cong \CC^{*} \ltimes \ZZ_{2}$ otherwise.
\end{enumerate}
\end{lemma}

\begin{proposition}[\protect{\cite[Proposition 4.6]{CM17}}]
$\bM_n\cong \bK_n(2,2)$ where $\bM_n^s$ corresponds to $\bK_n(2,2)^s$.
\end{proposition}

Consider Kirwan's partial desingularization of $\bK_n(2,2)= \bX\git G$ along the loci described in Lemma~\ref{lem:stability}. Let $\bY_{0} \subset \bX^{ss}$ be the locus of matrices equivalent to (1) of Lemma~\ref{lem:stability}. At each point $M = g\cdot \mathrm{Id} \in \bY_{0}$, the normal bundle $N_{\bY_{0}/\bX^{ss}}|_{M}$ is isomorphic to $H \otimes \mathfrak{sl}_{2}$ where $H \cong V^{*}/\langle g\rangle$. 
From Luna's slice theorem, there is a normal neighborhood of $\overline{M}\in\bX\git G$ that is isomorphic to
\[
H\otimes \mathfrak{sl}_{2}\git \mathrm{Stab}\; M \cong H\otimes \mathfrak{sl}_{2}\git \SL_{2},
\]
where $\SL_{2}$ acts on $\mathfrak{sl}_{2}$ in the standard way and $H$ in a trivial way. Also, $\ZZ_2$ acts trivially. Thus, from \cite[Lemma 3.11]{Kir85}, 
\begin{proposition}\label{bl1}
Let $\pi_{1} :{ \bX^{1}} \to \bX^{ss}:=\bX^{0}$ be the blow-up of $\bX^{0}$ along $\bY_{0}$. Then GIT quotient $\bX^{1}\git G$ is the blow-up of $\bX\git G$ along $\bY_0\git G\cong \PP^n$.
\end{proposition}

Let $\bY_{1} \subset \bX^{ss}$ be the locus of matrices equivalent to item (2) of Lemma~\ref{lem:stability}, where $\bY_{1}$ is a smooth variety and $\overline{\bY}_{1}$ is singular along $\bY_{0}$. Let $\bY_{1}^{1}$ be the proper transform of $\bY_{1}$ along the blow-up map $\pi_1$. At each point $M = \left[\begin{array}{cc}g&0\\0&k\end{array}\right] \in \bY_{1}$, the normal bundle $N_{\overline{\bY}_{1}/\bX^{ss}}|_{M}$ is isomorphic to $K \otimes \langle e, f\rangle$, where $K = V^{*}/\langle g, k\rangle$ and $\{h,e,f\}$ is the standard basis of $\mathfrak{sl}_{2}$. There is also a normal neighborhood of $\overline{M}\in\bX\git G$ isomorphic to 
\begin{equation}\label{singulars}
K \otimes \langle e, f\rangle \git \mathrm{Stab}\; M \cong K \otimes \langle e, f\rangle\git \CC^*,
\end{equation}
where $\CC^*$ acts on $\langle e, f\rangle$ with weights $2,-2$ and on $K$ in a trivial fashion. Also, $\ZZ_2$ acts trivially. 
Then $\overline{\bY}_{1}^{1}\git G$ is the blow-up $\overline{\bY}_{1}\git G\cong \PP^{n}\times \PP^n/\ZZ_2$ along the diagonal $\bY_0\git G\cong \PP^n$. Applying \cite[Lemma 3.11]{Kir85} again,
\begin{proposition}\label{bl2}
Let $\pi_{2} : {\bX^{2}} \to (\bX^{1})^{ss}$ be the blow-up of $(\bX^{1})^{ss}$ along $(\overline{\bY_{1}^{1}})^{ss}$. Then GIT quotient $\bX^{2}\git G$ is the  blow-up of $\bX^{1}\git G$ along $\overline{\bY_{1}^{1}}\git G\cong \mathrm{bl}_{\Delta}(\PP^{n}\times \PP^n/\ZZ_2)$.
\end{proposition}

Let $\overline{\pi}_{i}:\bX^{i}\git G \lr \bX^{i-1}\git G$ be the induced quotient map of $\pi_i$ for $i=1,2$. 
\begin{enumerate}
\item For $M \in \bY_{0}$,
\[
        \overline{\pi}_{1}^{-1}(\overline{M}) 
        \cong \PP(H\otimes \mathfrak{sl}_{2})\git \SL_{2}.
\]
The second blow-up $\overline{\pi}_{2}$ also provides partial a desingularization of $\PP(H \otimes \mathfrak{sl}_{2})\git \SL_{2}$ which is isomorphic to moduli space $\cM_{0}(\PP H, 2)$ (\cite[Theorem 4.1]{Kie07}). 

\item For $M \in \bY_{1}^{1}\setminus \bY_{0}^{1}$, 
\[
        \overline{\pi}_{2}^{-1}(\overline{M})
        \cong \PP^{n-2} \times \PP^{n-2}.
\]
\end{enumerate}

Thus, we have the following theorem.
\begin{theorem}[\protect{\cite[Theorem 5.1]{CM17}}]\label{desingularization}
The partial desingularization of $\bM_n$ is the second blown-up space $\bX^{2}\git G\cong \bK_n=\cM_0(\Gr(2, n+1), 2)$:
\[
\bX\git G\cong \bM_n\stackrel{\overline{\pi}_{1}}{\longleftarrow} \bX^{1}\git G:=\bM_n'\stackrel{\overline{\pi}_{2}}{\longleftarrow} \bX^{2}\git G=\bK_n.
\]
\end{theorem}

\subsection{(Intersection)  E-polynomial of a variety}
We review several polynomial invariants related to a quasi-projective variety $X$. All polynomial invariants defined in this section are related to the mixed Hodge structure on the compactly supported cohomology $\rH_{c}^*(X,\CC)$, the cohomology $\rH^*(X,\CC)$, the compactly supported intersection cohomology $\rIH_{c}^*(X,\CC)$, and intersection cohomology $\rIH^*(X,\CC)$. We use the notation $h_c^{p,q,i}$ for $\dim_{\CC} \mathrm{Gr}_F^p\mathrm{Gr}_{p+q}^W\rH_{c}^i(X,\CC)$, where $F^{\bullet}$ and $W_{\bullet}$ are Hodge and weight filtration respectively.

\begin{definition}
Let $X$ be a quasi-projective variety of dimension $\dim_{\CC}X=n$. 
The compactly supported Poincar\'{e}-Deligne  polynomial of $X$ can be expressed as 
\[
PD^c_X(u,v,t):=\sum_{p,q=0}^n \sum_{i=0}^{2n} h_c^{p,q,i} u^p v^q t^i ;
\]
the compactly supported  E-polynomial (or Serre polynomial) of $X$ as 
\[
E^c_X(u,v):=PD^c_X(u,v,-1)=\sum_{p,q=0}^n \sum_{i=0}^{2n} (-1)^i h_c^{p,q,i} u^p v^q ;
\]
the compactly supported Poincar\'{e} polynomial of $X$ as
\[
P^c_X(t):=PD^c_X(1,1,t)=\sum_{p,q=0}^n \sum_{i=0}^{2n} h_c^{p,q,i} t^i=\sum_{i=0}^{2n} \dim_{\CC}\rH_{c}^i(X,\CC) t^i ;
\]
and the virtual Poincar\'{e} polynomial of $X$ as
\[
P^{vir}_X(t):=PD^c_X(-t,-t,-1)=\sum_{p,q=0}^n \sum_{i=0}^{2n}(-1)^i h_c^{p,q,i}(-t)^{p+q}=\sum_{m=0}^{2n} \sum_{i=0}^{2n}(-1)^{i+m} \dim_{\CC} \mathrm{Gr}_{m}^W\rH_{c}^i(X,\CC) t^{m} .
\]
\end{definition}

The compactly supported  E-polynomial $E^c_X(u,v)$ of $X$ has a special property from the Grothendieck group of varieties $K_0(var)$ (or $K_0(var/pt)$ in a relative version). 
\begin{definition}
The Grothendieck group $K_0(var)$ of complex algebraic varieties is a free abelian group of isomorphism classes with an equivalence relation
$$ [X]=[Z]+[X\setminus Z] ,$$
where $Z$ is a Zariski closed subvariety in a variety $X$. It also has a ring structure with multiplication structure 
 $$ [X]\cdot[Y]=[X\times Y].$$ 
\end{definition}
For a quasi-projective variety $X$ of dimension $n$, the compactly supported cohomology group  $\rH_c^i(X,\QQ)$ carries a mixed Hodge structure, which induces the class
$$[\rH_c^*(X)]:=\sum_{i=0}^{2n}(-1)^i [\rH_c^i(X,\QQ)]$$
in the Grothendieck group $K_0(HS)$ of pure Hodge structures. There is a ring homomorphism $$[H_c^*]:K_0[var] \rightarrow K_0(HS)$$ 
defined by $[\rH_c^*]([X])=[\rH_c^*(X)]$. Also, there is another ring homomorphism (the \emph{Hodge-Euler polynomial})
$$E_{Hdg}: K_0(HS) \rightarrow \ZZ[u^\pm,v^\pm],\;E_{Hdg}([H])=\sum_{p,q} \left( \dim_{\CC} \mathrm{Gr}_F^p\mathrm{Gr}_{p+q}^W H_\CC \right) u^p v^q .$$
Hence the compactly supported  E-polynomial $E^c_X(u,v)$ of variety $X$ is nothing but $\left( E_{Hdg}\circ [H_c^*]\right) ([X])$.
\begin{notation}
$E_c (X):=E^c_X(u,v) \in \ZZ[u^\pm,v^\pm]$.
\end{notation} 

From the ring homomorphism $E_{Hdg}\circ [H_c^*]$, 
\begin{proposition}\label{proepoly}
\begin{enumerate}
\item $\rE_c(\CC^n)=(uv)^n.$
\item $\rE_c(X)=\rE_c(Z)+\rE_c(X\setminus Z)$ for any closed subset $Z\subset X$. \label{addition}
\item $\rE_c(X)=\rE_c(F)\cdot \rE_c(B)$ for the Zariski (resp. \'etale) locally trivial fibration $X\lr B$ with constant fiber $F$ (resp. $\Gr(k,n)$) (\cite[Lemma 3.1]{BJ12}). \label{multiplication}\label{productofE}
\end{enumerate}
\end{proposition}
In particular, the virtual Poincar\'{e} polynomial $P^{vir}_X(t)=E_c(X)(-t,-t)$ has similar properties to (\ref{addition}) and (\ref{multiplication}), called \emph{motivic properties}. We can define $PD_X(u,v,t)$, $E_X(u,v)$ and $P_X(t)$ from cohomology groups after replacing $h_c^{p,q,i}$ by $h^{p,q,i}:=\dim_{\CC} \mathrm{Gr}_F^p\mathrm{Gr}_{p+q}^W\rH^i(X,\CC)$, but the map $[H^*]:K_0[var] \rightarrow K_0(HS)$ is not homomorphism. Thus, motivic properties do not hold for any of the polynomial invariants from cohomology groups. However, there are a number of useful identities for calculation of these invariants. Let us shortly denote $E_X(u,v)$ by $E(X)$.
\begin{remark}\label{useful}
\begin{enumerate}
\item If $X$ is smooth and connected, then $\rE(X)(u,v)=u^n v^n \rE_c(X)(u^{-1},v^{-1})$ from the Poincar\'{e} duality.
\item If  $X$ is a compactification of an algebraic variety $U$, then $\rE_c(U)=\rE(X)-\rE(X\setminus U)$ (\cite[Section 5.5.2]{peters2008mixed}).\label{excisionofcompact}
\item If $\pi:(\tilde{X}, E)\to (X,D)$ is a proper modification with discriminant $D$, then $\rE(X)=\rE (\tilde{X})-\rE (E)+\rE (D)$ (\cite[Theorem 5.37]{peters2008mixed}). \label{resolution}
\item If $X$ is projective, then $\rE_c(X)=\rE(X)$. \label{projectivecase}
\item If $X$ is pure (i.e., $h_c^{p,q,i}=0$ for $p+q\neq i$) and balanced (i.e., $ h_c^{p,q,i}=0$ for $p\neq q$) type, then $P^c_X(t)=P^{vir}_X(t)$. \label{pureandbalanced}
\end{enumerate}
\end{remark}

\begin{definition}
The (compactly supported) intersection cohomology of a complex $n$-dimensional variety $X$ is defined by the hypercohomology
\begin{equation*}
\rIH_{(c)}^i(X,\QQ):= \HH_{(c)}^{i}(X,IC_X[-n]),
\end{equation*}
where $IC_X$ is the \emph{intersection complex} on $X$ of the middle perversity.
\end{definition}
Since $\rIH_{c}^i(X,\QQ)$ carries a mixed Hodge structure, we define the (compactly supported) intersection cohomology  E-polynomial
$$\rIE_{(c)} (X)=E_{Hdg}\circ [IH_{(c)}^*]$$ 
where the map $[IH_{(c)}^*]$ is defined by $X  \mapsto [IH_{(c)}^*(X)]:=\sum\limits_{i=0}^{2n}(-1)^i [\rIH_{(c)}^i(X,\QQ)] \in K_0[HS]$ for a quasi-projective variety $X$. Unfortunately, $[IH_{(c)}^*]$ does not provide group homomorphism from $K_0(var)$ in general. 

The (compactly supported) cohomlogy $\rH_{(c)}^i(X,\QQ)$ and the (compactly supported) intersection cohomology $\rIH_{(c)}^i(X,\QQ)$ of variety $X$ are not isomorphic  in general when $X$ is a singular variety. However, there is an isomorphism for some special cases. 
\begin{proposition}\label{rationalsmooth}(\cite[Theorem 6.6.3]{maxim2018intersection} and \cite[Proposition A.1]{Bri99})
If a variety $X$ has at most finite quotient singularities (more generally, rationally smooth manifold), then $IC_X$ is quasi-isomorphic to $\QQ_X[\dim X]$. In particular, $$\rIH_{(c)}^*(X,\QQ)=\rH_{(c)}^*(X,\QQ).$$
\end{proposition}

\subsection{Comparison via algebraic stratification}
We give a relationship between polynomial invariants of $X$ from $\rH_{(c)}^i(X,\QQ)$ and $\rIH_{(c)}^i(X,\QQ)$. We give a more general statement ahead. Let $f:X\lr Y$ be a proper morphism of complex algebraic varieties. We fix a complex \emph{algebraic Whitney stratification} $\cV$ of $f$ such that all strata of $X$ and $Y$ are smooth and $f$ is a stratified submersion. This satisfies the \emph{frontier condition}: if $W \cap \bar{V} \neq \emptyset $, then $W \subset \bar{V}$. We have partial order in $\cV$ by $W\leq V$ if $W \subset \overline{V}$ and $W<V$ for $\dim (W)<\dim(V)$. Let $Y^{\circ}$ be a dense open stratum in $Y$, then $Y^{\circ}$ is the maximal element in $\cV$. For each order pair $W < V$ and $w\in W$, consider a local analytic embedding $(V,w) \hookrightarrow (\CC^n,0)$ of neighborhood $(V,w)$ of $w$.  Let $N$ be a smooth, normal neighborhood of $w$ transversally meeting with $W$ only at $w$ and $\dim N=\mathrm{codim}_{\CC^n} W$. Let $L_{w,V}:=V\cap N\cap \partial B_{\delta}(w)$, where $B_{\delta}(w)$ is an open ball in $\CC^n$ with radius $0<\delta \ll 1$ centered at $w$. Then the stratification satisfies
\begin{itemize}
\item the open cone $c^0L_{w,V}:=\left(L_{w,V}\times [0,1)\right) /\left(L_{w,V}\times\{0\}\right)$ is homeomorphic to $V\cap N\cap  B_{\delta}(w)$, and
\item the homeomorphic type of $L_{w,V}$ does not depend on the choice of $w\in W$.
\end{itemize}

\begin{notation}
We call $L_{w,V}$ the \emph{link} $L_{W,V}$ of $W$ in $V$ and the open cone of $L_{W,V}$ is denoted as $c^{\circ}L_{W,V}$. We denote $c^{\circ}L_{W,V}$ by $c^{\circ}L_{W,\bar{V}}$ unless stated otherwise.
\end{notation}

Since the open cone has canonical mixed Hodge structure, we can define the  E-polynomials $\rE(c^\circ L_{W,V})$ and $\rIE(c^\circ L_{W,V})$. The difference between $\rE(X)$ and $\rIE(X)$ for a variety $X$ can be obtained from \cite{MR2376848}.
\begin{proposition}\label{comparisonhodgestructure}
Let $f:X\lr Y$ be a proper morphism between algebraic varieties. Fix an algebraic stratification of $Y$ satisfying the above conditions, and assume that $f$ induces a trivial fibration on each stratum. Then, 
\begin{equation}\label{comparisonhodgestructure1}
[H_{(c)}^*(X)]=[IH_{(c)}^*(Y)]\cdot [H^*(F)]+ \sum_{V< Y^{\circ}}\widetilde{[IH_{(c)}^*(\overline{V})]}\cdot \left( [H^*(F_V)]-[H^*(F)]\cdot [IH^*(c^\circ L_{V,Y})] \right),
\end{equation}
where $\widetilde{[IH_{(c)}^*(\overline{V})]}$ is inductively defined by 
\begin{equation}\label{comparisonhodgestructure2}
\widetilde{[IH_{(c)}^*(\overline{V})]}:= [IH_{(c)}^*(\overline{V})]- \sum_{W<V} \widetilde{[IH_{(c)}^*(\overline{W})]}\cdot [IH^* (c^\circ L_{W,V})],
\end{equation}
and $F$ (resp. $F_V$) is the fiber over $Y^{\circ}$ (resp. $V\in \cV$).
\end{proposition}

\begin{proof}
Let $M=\QQ_X^H$ and apply $k'_*$ ($k'_!$ for compactly supported cases) to the identity of Corollary 3.4 in \cite{MR2376848}, where $k':Y\lr \mathrm{pt}$. Classes $[H^*(F)]$, $[H^*(F_V)]$, $[IH^*(c^\circ L_{V,Y})]$, and $[IH^* (c^\circ L_{W,V})]$ are in the Grothendieck group $K_0(MHM(pt))$ of mixed Hodge modules on a point. Since $k'_*$ ($k'_!$) is a $K_0(MHM(pt))$-linear map, we obtain the result.
\end{proof}

We drop subscript $c$ for compactly supported cohomologies in the remainder of this paper. However, $\rIE(c^\circ L_{A,B})$ of any open cone $c^\circ L_{A,B}$ always refers to the intersection cohomology, rather than the compactly supported intersection cohomology.

\begin{corollary}\label{mainprop}
Under assumptions of Proposition~\ref{comparisonhodgestructure}, 
\[
\rE_{}(X)=\rIE_{}(Y)\cdot \rE(F)+ \sum_{V< Y^{\circ}}\widetilde{\rIE_{}}(\overline{V})\cdot(\rE(F_V)-\rE(F)\cdot \rIE(c^\circ L_{V,Y})),
\]
where $\widetilde{\rIE_{}}(\overline{V})$ is inductively defined by $$\widetilde{\rIE_{}}(\overline{V}):= \rIE_{}(\overline{V})- \sum_{W<V} \widetilde{\rIE_{}}(\overline{W})\cdot \rIE (c^\circ L_{W,V}), $$
and $F$ (resp. $F_V$) is the fiber over $Y^{\circ}$ (resp. $V\in \cV$).
\end{corollary}

\begin{proof}
The identity of the claim is obtained by applying ring isomorphism $E_{Hdg}$ in \eqref{comparisonhodgestructure1} and \eqref{comparisonhodgestructure2} respectively.
\end{proof}

\begin{corollary}\label{comparisoneandie}
\[
\rE_{}(Y)=\rIE_{}(Y)+ \sum_{V< Y^{\circ}}\widetilde{\rIE_{}}(\overline{V})\cdot(1-\rIE(c^\circ L_{V,Y}))
\]
for any stratification $\cV$.
\end{corollary}

\begin{proof}
The result follows immediately by letting $X=Y$ and $f=\mathrm{id}$ in Corollary~\ref{mainprop}.
\end{proof}

\begin{example}\label{exam1}
Let $X$ be a smooth projective variety. The symmetric product $\mathrm{Sym}^2(X)$ has the $\ZZ_2$-quotient singularity along the diagonal $\Delta(\cong X)\subset \mathrm{Sym}^2(X)$, and hence $\rE(\mathrm{Sym}^2(X))=\rIE(\mathrm{Sym}^2(X))$ from Proposition~\ref{rationalsmooth}. Applying Corollary~\ref{comparisoneandie} for the stratification $\cV=\{\Delta, V:=\mathrm{Sym}^2(X)\setminus \Delta\}$, 
\[
\rIE(c^\circ L_{\Delta, V})=1.
\]
\end{example}

\begin{proposition}\label{coneofquadric}
Let $Q:=\{xy-zw=0\}\subset \CC^4$ be the quadric cone in $\CC^4$. Then the IE-polynomial of the open cone of the link $Q$ at the origin $(0,0,0,0)\in \CC^4$ is given by 
\[
\rIE(c^\circ L_{\{0\},Q})=uv+1.
\]
\end{proposition}

\begin{proof}
Let $\bar{Q}=\{xy-zw=0\} \subset \PP^4$ be the closure of $Q$ under standard embedding $\CC^4\subset \PP^4, (x,y,z,w)\mapsto [x:y:z:w:1]$. Let $P=[0:0:0:1]$ be the singular point of $\bar{Q}$. There are two resolutions $\tilde{Q}_1$ and $ \tilde{Q}_2$ for the singular point $P$. The exceptional divisor of the resolution $\tilde{Q}_1\lr \bar{Q}$ is $\PP^1\times \PP^1$ and $\tilde{Q}_1$ is a $\PP^1$-bundle over $\PP^1\times \PP^1$. Thus,
$$\rE(\bar{Q})=\rE(\tilde{Q}_1)-\rE(\PP^1\times\PP^1)+\rE(P)=(uv)^3+2(uv)^2+(uv)+1.$$

The resolution $\tilde{Q}_2\lr \bar{Q}$ is small with the  exceptional locus $\PP^1$, hence
$$\rIE(\bar{Q})=\rE(\tilde{Q}_2)=\rE(\bar{Q})-\rE(P)+\rE(\PP^1)=(uv)^3+2(uv)^2+2(uv)+1;$$
and from Corollary~\ref{mainprop},
$$\rIE(c^\circ L_{\{0\},Q})=uv+1.$$
\end{proof}

We use the notation $\rP(X):=\rE(X)(-1,-1)=P^{vir}_X(t)$ and $\rIP(X):=\rIE(X)(-1,-1)$, where $\rIP(c^\circ L_{A,B})$  comes from the usual intersection cohomology of  the open cone $c^\circ L_{A,B}$.

\section{Proof of main result}\label{sec:mainresults}

This section proves Theorem~\ref{mainthm1}. The intersection Poincar\'e polynomial of $\bM_n$ is exactly the same as the virtual one for $\bM_n$ because all spaces arising in the computation are pure and balanced Hodge types. We propose a numerical relationship between the virtual intersection Poincar\'e polynomials of open cones of singular loci in $\bM_n$.

\subsection{Intersection cohomology of $\bM_n$ using Kirwan's resolution} 

For a pure dimensional variety $X$, let us write the \emph{truncated} intersection Poincar\'e polynomial of $X$ as 
$$\rIP(X)_{<k}:=\sum_{i=0}^{k-1} \dim \rIH^i(X,\QQ)t^i.$$ 
\begin{lemma}\label{mainlemma}
The intersection Poincar\'e polynomial of the GIT-quotient space $\PP(\mathrm{Sym}^2\CC^2\otimes \CC^n)\git\mathrm{SL}(2)$ is
\[
\frac{(1-t^{2n})[(1-t^{2n+2})(1-t^{2n-2})-t^2(1-t^{4\lfloor\frac{n-1}{2}\rfloor})(1-t^{4\lfloor\frac{n}{2}\rfloor})]}{(1-t^2)^2(1-t^4)},
\]
where $\lfloor x\rfloor$ is the largest integer $\leq x$.
\end{lemma}

\begin{proof}
Since the blow-up of the stable maps space $\cM_{0}(\PP^{n-1}, 2)$ along a $\PP^2$-bundle over $\Gr(2,n)$ is isomorphic to the blow-up of the Hilbert scheme $\bH(\PP^{n-1})$ of conics along a $\PP^2$-bundle over $\Gr(3,n)$ (see \cite[Section 4]{Kie07} and \cite[Section 3]{CHK12} for an explicit geometric description),
\begin{equation}\label{moduliofstablemaps}
\rP(\cM_{0}(\PP^{n-1}, 2))=\rP(\bH(\PP^{n-1}))-\rP(\PP^2)\rP(\Gr(3,n))+\rP(\PP^2)\rP(\Gr(2,n)).
\end{equation}
Kirwan's partial desingularization of $\PP(\mathrm{Sym}^2\CC^2\otimes \CC^n)\git\mathrm{SL}(2)$ along the strictly semi-stable locus $\PP(\mathrm{Sym}^2\CC^2)\times \PP(\CC^n)\git\SL(2)\cong \PP^{n-1}$ is isomorphic to the moduli space $\cM_{0}(\PP^{n-1}, 2)$, where the exceptional locus is a $\mathrm{Sym}^2\PP^{n-2}$-fibration over $\PP^{n-1}$ (\cite[Theorem 4.1]{Kie07}). 
Since $\pi_1(\PP^{n-1})=0$, we can apply \cite[2.28]{Kir86a}, 
\begin{equation}\label{moduliofstablemaps1}
\rIP(\PP(\mathrm{Sym}^2\CC^2\otimes \CC^n)\git\mathrm{SL}(2))=\rP(\cM_{0}(\PP^{n-1}, 2))-\rP(\PP^{n-1})\cdot(t^2Q(t)+t^{4n-8}Q(\frac{1}{t}) ),
\end{equation}
where $Q(t):=\rIP(\mathrm{Sym}^2\PP^{n-2})_{\small{<2n-4}}$. Note that $\rP(\cM_{0}(\PP^{n-1}, 2))=\rIP(\cM_{0}(\PP^{n-1}, 2))$ from Proposition~\ref{rationalsmooth}. The Hilbert scheme $\bH(\PP^{n-1})$ is also isomorphic to a $\PP^5$-bundle over $\Gr(3,n)$ and $\rIP(\mathrm{Sym}^2\PP^{n-2})=\frac{1}{2}((\frac{1-t^{2n-2}}{1-t^2})^2+\frac{1-t^{4n-4}}{1-t^4})$ (\cite[Lemma 2.6]{MOG09}). 

Thus the result follows from \eqref{moduliofstablemaps} and \eqref{moduliofstablemaps1}.
\end{proof}

Mart\'in proved the following proposition using torus localization method (\cite[Theorem 3.1]{LM14}). We use an alternative birational geometric proof to confirm that the related moduli spaces are pure and balanced types.
\begin{proposition}\label{mainprop2}
The Poincar\'e polynomial of $\bK_n=\cM_0(\mathrm{Gr}(2,n+1),2)$ is
\[
        \frac{[(1+t^{2n+2})(1+t^6)-t^2(1+t^2)(t^4+t^{2n-2})]
        (1-t^{2n+2})(1-t^{2n})(1-t^{2n-2})}{(1-t^2)^3(1-t^4)^2}.
\]
\end{proposition}

\begin{proof}
Let $\bH(\Gr(2,n+1))$ be the Hilbert scheme of conics in $\Gr(2,n+1)$. Let $\Gr(2,\cU)$ be the Grassmannian bundle over the universal sub-bundle $\cU$ of $\Gr(4,n+1)$. Consider the relative Hilbert scheme $\bH(\Gr(2,\cU))$ of conics over $\Gr(4,n+1)$. Following the method used to prove Proposition 4.2 in \cite{CHL18}, the natural forgetful map $\bH(\Gr(2,\cU))\lr \bH(\Gr(2,n+1))$ is a blow-up map along a $\PP^5$-bundle over $\Gr(3,n+1)$. 

On the other hand, the space $\bK_n$ is a blow-up of space $\bH(\Gr(2,n+1))$ along a $\PP^{n-2}$-bundle over $\Gr(1,3,n+1)$ followed by blowing-down along a $\PP^2$-bundle over $\Gr(1,3,n+1)$ (\cite[Corollary 5.3]{CHK12}). Therefore, 
\[
\rP(\bK_n)=\rP(\bH(\Gr(2,\cU)))-\rP(\PP^5) \rP(\Gr(3,n+1)) (\rP(\PP^{n-3})-1)+ \rP(\Gr(1,3,n+1)) (\rP(\PP^{2})-\rP(\PP^{n-2})),
\]
since $\bH(\Gr(2,4))$ is the blow-up of $\Gr(3,6)$ along two copies of disjoint $\PP^5$. Note that from Proposition \ref{proepoly} (3) and \cite[Lemma 2.1]{MOG09}, the relevant spaces are pure and balanced type.
\end{proof}

\begin{proof}[Proof of Theorem~\ref{mainthm1}]
Applying \cite[(2.28)]{Kir86a} to the first blow-up of Proposition~\ref{bl1}, 
\[
\rIP(\bM_n)=\rIP(\bM_n')-\rP(\PP^n)[t^2R(t)+t^{8n-6}R(\frac{1}{t})-ih^{3n-5}(\PP(\mathrm{Sym}^2\CC^2\otimes \CC^n)\git \mathrm{SL}(2))t^{4n-2}],
\]
where 
\[\begin{split} 
R(t)&=\rIP(\PP(\mathrm{Sym}^2\CC^2\otimes \CC^n)\git\mathrm{SL}(2))_{3n-4}\; \mathrm{and} \\
\;ih^{3n-5}(\PP(\mathrm{Sym}^2\CC^2\otimes& \CC^n)\git \mathrm{SL}(2))=\dim \rIH^{3n-5}(\PP(\mathrm{Sym}^2\CC^2\otimes \CC^n)\git \mathrm{SL}(2), \QQ).
\end{split}\]

Applying \cite[(2.1)]{Kir86a} to the second blow-up of Proposition~\ref{bl2} (\cite[Section 5]{Kir86b}, \cite[Lemma 2.6]{MOG09} and \cite[Remark 2.7]{LMN13}),
\[\begin{split}
\rIP(\bM_n')=\rP(\bK_n)&-\frac{1}{2}[\rP(\PP^n)^2+\frac{1-(-t^2)^{2n+2}}{1-t^4}+2\cdot\rP(\PP^n)\cdot\frac{t^{2n}-t^2}{t^2-1}]\cdot\sum_{j=1}^{2n-4}\lfloor\frac{\mathrm{min}\{j+1,2n-2-j\}}{2}\rfloor t^{2j}\\
&-\frac{1}{2}[\rP(\PP^n)^2-\frac{1-(-t^2)^{2n+2}}{1-t^4}]\cdot\sum_{j=2}^{2n-5}\lfloor\frac{\mathrm{min}\{j,2n-3-j\}}{2}\rfloor t^{2j},
\end{split}\]
where $\lfloor x\rfloor$ is the largest integer $\leq x$.

The proof follows by combining these two relation with Lemma~\ref{mainlemma} and Proposition~\ref{mainprop2}.
\end{proof}

\begin{remark}
$\bM_2\cong\bK_2(2,2)\cong \PP^5$ (\cite[Corollary 4.3]{LP93a}) and hence $\rIP(\bM_2)=1+t^2+t^4+t^6+t^8+t^{10}$, which is consistent with Theorem~\ref{mainthm1}.
\end{remark}

The following proposition is used subsequently in the paper. Recall that $\Delta:=\bY_0\git G=\PP^n$ and $S_n:=\overline{\bY}_1\git G\cong \mathrm{Sym}^2(\PP^n)$ are the blow-up centers of $\bM_n$.
\begin{corollary}\label{univrelation}
For the stratification $\cV=\{\Delta, S_n\setminus \Delta, \bM_n\setminus S_n\}$ of $\bM_n$, the intersection Poincar\'e polynomials of open cones are related by
\begin{equation}\label{relation}\begin{split}
(t^2+1)&\rIP(c^\circ L_{\Delta,\bM_n})+t^4\frac{1-t^{2n}}{1-t^2}\rIP(c^\circ L_{S_n\setminus \Delta,\bM_n})\\
&=\frac{(t^{2n-2}-1)(2t^{2n+4}-t^4-1)+\frac{(-1)^n+1}{2}\cdot(t^{2n-2}+t^{4n})(1-t^2)^2}{(t^2-1)(t^4-1)}.
\end{split}\end{equation}
\end{corollary}

\begin{proof}
The result follows from Corollary~\ref{comparisoneandie}, Theorem~\ref{mainthm1}, Example~\ref{exam1}, and Proposition 7.2 of \cite{CM17}\footnote{There is a small error on page 648 of \cite{CM17}. In (7.1), the term $ \left(P((\PP^{n-2})^{2})-1\right)
        \left(\frac{1}{2}\left(P(\PP^{n})^{2}+\frac{1-q^{2n+2}}{1-q^{2}}
        \right)-P(\PP^{n})\right)$ must be replaced by $\rP(\Sym^2(\PP^n\times \PP^{n-2}))-\rP(\PP^n)\cdot\rP(\Sym^2\PP^{n-2})-(\rP(\Sym^2\PP^n)-\rP(\PP^n))$ (\eqref{singulars} and \cite[Proposition 2.6]{LMN13}).}.
\end{proof}

\begin{example}
When $n=3$, the intersection cohomology for each open cone can be calculated using \eqref{relation} and Proposition~\ref{coneofquadric}. From Luna's slice theorem, the space $\bM_n$ at $[E]=[\cO_H\oplus \cO_{H'}]$, $H\neq H'$ is locally isomorphic to
\[
\Ext^1(E,E)\git \mathrm{Aut}(E) \cong (\Ext^1(\cO_H,\cO_H)\oplus \Ext^1(\cO_{H'},\cO_{H'})) \times Y\subset (\CC^n\oplus \CC^n)\times \CC^{(n-1)^2},
\]
where $S_n$ at $[E]$ corresponds to the affine space $(\Ext^1(\cO_H,\cO_H)\oplus \Ext^1(\cO_{H'},\cO_{H'}))\times\{0\}$ at the origin $\{(0\oplus0)\times\{0\}\}$, and $Y$ is isomorphic to the affine cone of the Segre variety $\PP^{n-2}\times \PP^{n-2}\subset \PP^{n^2-2n}$ (\eqref{singulars} and \cite[Proposition 7.16]{Dre04}). 

For the case $n=3$, let us choose the normal slice $N=\{(0\oplus 0)\}\times \CC^{4}$ at $[E]$. Then,
\[
c^\circ L_{S_3, \bM_3}\cong c^\circ L_{\{0\},\;Q}.
\]
Substituting $\rIP(c^\circ L_{S_3,\;\bM_3})=1+t^2$ into \eqref{relation}, $\rIP(c^\circ L_{\Delta,\;\bM_3})=1$.
\end{example}

\section{Application to local surfaces}\label{ihdoflocalsurface}

This section calculates the intersection Poincar\'e polynomial of the moduli space of pure one-dimensional sheaves on del Pezzo surfaces ($\FF_0$, $\mathbb{F}_1$, and $\PP^2$) using the intersection Poincar\'e polynomial of the space $\bM_n$ ($3\leq n\leq 5$). 

Recall that $\bM_S(c, \chi)$ is the moduli space of semi-stable sheaves $F$ with $c_1(F)=c$ and $\chi(F)=\chi$ on a del Pezzo surface $S$. From the Serre duality and the semi-stability of $F$, $\Ext_S^2(F,F)\cong \Ext_S^0(F,F\otimes K_S)=0$, and hence the Quot scheme arising in the GIT-construction of $\bM_S(c, \chi)$ is smooth. Therefore, the analytic neighborhood of the singular locus in $\bM_S(c, \chi)$ is isomorphic to that of vector bundles case (cf. Remark~\ref{luna}).

For $n=3$ and $5$, we can use the explicit birational maps between spaces $\bM_S(c, \chi)$ and $\bM_n$ (see \cite[Theorem 5.7]{CM16} and \cite[Proposition 7.4]{CM17}). However, we use Corollary~\ref{univrelation} for $n=4$ since we do not know any explicit birational relation between $\bM_S(c, \chi)$ and $\bM_n$. Lastly, we conjecture that the intersection Poincar\'e polynomial of the moduli space $\bM_S(c,\chi)$ does not depend on the Euler characteristic $\chi$.

\subsection{Cases $n=3$ and $5$}

Let $\bM_{\FF_0}((2,2),2)=\bM_{\cO_{\FF_0}(1,1)}(\FF_0, (2,2), 4m+2)$ be the moduli space of semi-stable sheaves $F$ with $c_1(F)=(2,2)\in \rH_2(\FF_0,\ZZ)$ with Hilbert polynomial $P(F)(m)=4m+2$.
\begin{corollary}\label{cor1}
Let $\bM_{\FF_0}((2,2),2)$ be the moduli space of pure sheaves on $\FF_0$. Then the intersection Poincar\'e polynomial of $\bM_{\FF_0}((2,2),2)$ is 
\[
1+3t^2+4t^4+4t^6+4t^8+4t^{10}+4t^{12}+4t^{14}+3t^{16}+t^{18}.
\]
\end{corollary}

\begin{proof}
From Theorem 5.7 of \cite{CM16}, there exists a birational morphism $$\bM_{\FF_0}((2,2),2)\longrightarrow \bM_3$$ that is a smooth blow-up at two distinct smooth points. From the blow-up formula for cohomology groups, 
\[
\rIP(\bM_{\FF_0}((2,2),2))=\rIP(\bM_3)+2\cdot\rP(\{\mathrm{pt}\})(\rP(\PP^{\dim \bM_3-1})-1),
\]
and $\rP(\PP^n)=\frac{1-t^{2n+2}}{1-t^2}$ the result follows from Theorem~\ref{mainthm1}.
\end{proof}

\begin{remark}\label{imrem1}
The space $\bM_{\FF_0}((2,2),1)$ is isomorphic to the relative Hilbert scheme of one point over the complete linear system $|\cO_{\FF_0}(2,2)|$ (\cite[Proposition 12]{BH14}), where the latter space is isomorphic to a $\PP^7$-bundle over $\FF_0$.
Thus, from Proposition~\ref{proepoly}, $\rIE(\bM_{\FF_0}((2,2),1))=\rIE(\PP^7\times \PP^1 \times \PP^1)$, which is $\rIE(\bM_{\FF_0}((2,2),2))$.
\end{remark}

\begin{corollary}\label{cor3}
Let $\bM_{\PP^2}(4,2)$ be the moduli space of semi-stable sheaves on $\PP^2$ with the Hilbert polynomial $4m+2$. Then the intersection Poincar\'e polynomial of $\bM_{\PP^2}(4,2)$ is 
\[\begin{split}
&1+2t^2+6t^4+10t^6+14t^8+15t^{10}+16t^{12}+16t^{14}+16t^{16} \\
&+16t^{18}+16t^{20}+16t^{22}+15t^{24}+14t^{26}+10t^{28}+6t^{30}+2t^{32}+t^{34}.
\end{split}\]
\end{corollary}

\begin{proof}
The spaces $\bM_{\PP^2}(4, 2)$ and $\bM_5$ are related by Bridgeland wall-crossings on $\PP^2$ (\cite[Section 6]{BMW14} and \cite[Proposition 7.4]{CM17}) with wall-crossing loci given in Table~\ref{tbl:wall}.
\begin{table}[!ht]
\begin{tabular}{|c|c|}
\hline
First wall ($W_1$) & Second wall ($W_2$)\\
\hline\hline
$\ses{\cO_{\PP^2}(1)}{F}{\cO_{\PP^{2}}(-3)[1]}$ & $\ses{I_p(1)}{F}{I_q^{\vee} (-3)[1]}$ for $p$ and $q\in \PP^2$\\
\hline
$\ses{\cO_{\PP^{2}}(-3)[1]}{F'}{\cO_{\PP^2}(1)}$ &$\ses{I_q^{\vee} (-3)[1]}{F'}{I_p(1)}$ for $p$ and $q\in \PP^2$ \\
\hline
\end{tabular}
\medskip
\caption{Bridgeland wall-crossings between $\rM_{\PP^{2}}(4,2)$ and $\bM_5$}
\label{tbl:wall}
\end{table}
Since $\dim \Ext^1(F,F)=17=\dim \bM_{\PP^2}(4,2)$ for $F, F'\in W_1$ or $W_2$, the wall-crossing loci are contained in the smooth part of moduli spaces. The result follows by comparing intersection cohomology groups.
\end{proof}

\begin{remark}\label{imrem2}
The intersection Poincar\'e polynomial of $\bM_{\PP^2}(4,2)$ is exactly the same as that of $\bM_{\PP^2}(4,1)$ (\cite[Corollary 5.2]{CC17}).
\end{remark}
\subsection{Case $n=4$}

Let $\mathbb{F}_1=\PP(\cO_{\PP^1}\oplus \cO_{\PP^1}(-1))$. Consider the blow-up map $\FF_1\lr \PP^2$ at a point. $\rH_2(\mathbb{F}_1,\ZZ)\cong \ZZ\cdot h \oplus \ZZ\cdot e$, where $h$ is the hyperplane class and $e$ is the exceptional divisor class. The canonical divisor of $\FF_1$ is $K_{\FF_1}=-3h+e$ and the arithmetic genus of curve $C$ in $\FF_1$ with $c_1(\cO_C)=dh-ne$ is 
\begin{equation}\label{deggenus}
p_a(C)=\frac{(d-1)(d-2)}{2}-\frac{n(n-1)}{2}
\end{equation}
from the adjunction formula. 

Let $\bM_{\FF_1}((4,2),2)=\bM_{K_{\FF_1}^*}(\FF_1, (4,2), 10m+2)$ be the moduli space of semi-stable sheaves $F$ with $c_1(F)=4h-2e\in \rH_2(\FF_1,\ZZ)$ and the Hilbert polynomial $P(F)(m)=10m+2$.
To obtain the Poincar\'e polynomial of $\bM_{\FF_1}((4,2), 2)$, we use the wall-crossings of the moduli space of $\alpha$-stable pairs. For the ample line bundle $L = -K_{\FF_1}$, let $P(F)(m)=\chi(F\otimes L^{m})$ be the Hilbert polynomial of a coherent sheaf $F$ on $\FF_1$.
A pair $(s, F)$ consists of a coherent sheaf $F$ on $\FF_1$ and a nonzero section $\cO_{\FF_1} \stackrel{s}{\to} F$. The pair is $\alpha$-\emph{semi-stable} if $F$ is pure and, for any subsheaf $F'\subset  F$ 
\[
        \frac{P(F')(m)+\delta\cdot\alpha}{r(F')} \le
        \frac{P(F)(m)+\alpha}{r(F)}
\]
holds for $m\gg 0$, where $r(F)=-K_{\FF_1}\cdot c_1(F)$ and $\delta=1$ if the section $s$ factors through $F'$ and $\delta=0$ otherwise. When the strict inequality holds, $(s,F)$ is called an $\alpha$-\emph{stable} pair.

There exists a projective scheme $\bM_{L}^{\alpha}(\FF_1,P(m))$ parameterizing $S$-equivalence classes of $\alpha$-semi-stable pairs with the Hilbert polynomial $P(m)$ (\cite[Theorem 2.6]{He98}). We also have a decomposition of the moduli space
\[
        \bM_{L}^{\alpha}(\FF_1,P(m))=\bigsqcup_{\beta\in H_2(\FF_1, \ZZ)}
        \bM_{L}^{\alpha}(\FF_1,\beta, P(m)).
\]
\begin{notation}
We denote $\bM_{L}^{\alpha}(\FF_1,\beta, P(m))$ by $\bM_{\FF_1}^\alpha(\beta, P(0))$. If $\alpha$ is sufficiently large (resp. small), we denote  $\alpha=\infty$ (resp. $\alpha = 0^{+}$).
\end{notation}
Wall-crossing phenomena of moduli spaces $\bM_{\FF^1}^\alpha((4,2),2)$ can be analyzed by the following propositions.
\begin{proposition}[\protect{\cite{BJRR12}}]\label{secnum}
\item Let $h^i(m,n):=\dim \rH^i(\FF_1, \cO_{\FF_1}(mh-ne))$. Then, 
\begin{enumerate}
\item $h^0(m,n)={{m+2}\choose{2}}-{{n+1}\choose{2}}$
\item \begin{equation*}
h^1(m,n)=
    \begin{cases*}
       {{-n}\choose{2}}-{{-m-1}\choose{2}} & if $m\geq n$ and $-2\geq n$, \\
      {{n+1}\choose{2}}-{{m+2}\choose{2}}       & if $m-n\leq -2$ and $1\leq n$,\\
      0& otherwise.
    \end{cases*}
  \end{equation*}
  \item 
  \begin{equation*}
h^2(m,n)=
    \begin{cases*}
       {{-m-1}\choose{2}}-{{-n}\choose{2}} & if $m\leq 0$ and $n\leq 0$, \\
      0& otherwise,
    \end{cases*}
  \end{equation*}  
\end{enumerate}
where ${{r}\choose{2}}:=0$ for $r<2$.
\end{proposition}

\begin{proposition}[\protect{\cite[Corollary 1.6]{He98}}]\label{defcoh} 
Let $\Lambda=(s, F)$ and $\Lambda'=(s', F')$ be pairs on a smooth projective variety $X$. Then, there exists a long exact sequence
\begin{align*}
0&\lr \Hom(\Lambda,\Lambda')\lr \Hom (F,F')\lr \Hom(s,H^0(F')/s')\\
&\lr \Ext^1(\Lambda,\Lambda')\lr \Ext^1(F,F')\lr \Hom(s,H^1(F'))\\
&\lr \Ext^2(\Lambda,\Lambda')\lr \Ext^2(F,F')\lr \Hom(s,H^2(F'))\lr \cdots.
\end{align*}
\end{proposition}

\begin{proposition}\label{inftystable}
The $\infty$-stable pairs space $\bM_{\FF^1}^\infty((4,2),2)$ is a $\PP^8$-bundle over the Hilbert scheme of three points on $\FF_1$.
\end{proposition}

\begin{proof}
From Proposition~\ref{secnum}, $h^0(\cO_{\FF_1}(4h-2e))=12$ and $h^1(\cO_{\FF_1}(4h-2e))=0$. The line bundle $\cO_{\FF_1}(4h-2e)$ is also $2$-very ample from \cite{DR96}. Thus, the result follows by applying the same argument as Lemma 2.3 in \cite{CC17}.
\end{proof}

During wall-crossings, a pair $(s, F)$ is lying in the wall at $\alpha$ if 
\[
(s,F)=(s,F_1)+(0,F_2)\; \text{and}
\]
\[
\frac{\chi(F)+\alpha}{K_{\FF_1}\cdot c_1(F)}=\frac{\chi(F_1)+\alpha}{K_{\FF_1}\cdot c_1(F_1)}=\frac{\chi(F_2)}{K_{\FF_1}\cdot c_1(F_2)}.
\]
The numerical walls of $\bM_{\FF^1}^\alpha((4,2),2)$ are listed in Table~\ref{table2} by a direct calculation.
\begin{table}
\begin{tabular}{|l|p{9cm}|}
\hline
$\alpha$
&
$(s,((4,2),2))=(s,(d_1h-n_1e,\chi_1))\oplus (0,(d_2h-n_2e,\chi_2))$\\
\hline
\hline
$\frac{1}{2}$&
$(s,((2,0),1))\oplus (0,((2,2),1))$ \\
\hline
$\frac{4}{3}$&
$(s,((2,2),0))\oplus (0,((2,0),2))$ \\
\hline
$\frac{4}{3}'$&
$(s,((3,2),1))\oplus (0,((1,0),1))$\\
\hline
$3$&
$(s,((3,1),1))\oplus (0,((1,1),1))$\\
\hline
$8$&
$(s,((3,1),0))\oplus (0,((1,1),2))$\\
\hline
$8'$&
$(s,((4,3),1))\oplus (0,((0,-1),1))$\\
\hline
$13$&
$(s,((3,1),-1))\oplus (0,((1,1),3))$\\
\hline
\end{tabular}
\medskip
\caption{Numerical walls of $\bM_{\FF_1}^{\alpha}((4,2),2)$}
\label{table2}
\end{table}

\begin{lemma}
The walls at $\alpha=\frac{1}{2},\frac{4}{3}$, and $13$ (Table~\ref{table2}) are empty.
\end{lemma}

\begin{proof}
The case $\alpha=\frac{1}{2}$ cannot occur. Let $F_2$ be a stable sheaf with $c_1(F_2)=2h-2e$ and $\chi(F_2)=1$. Support of $F$ should be the fiber of the projection map $p:\FF_1=\PP(\cO_{\PP^1}\oplus \cO_{\PP^1}(-1))\lr \PP^1$. From $\chi(F_2)=1$, we have the canonical map $s:\cO_{\FF_1}\lr F_2$. Hence the image of $s$ is of the form $\mathrm{im}(s)=\cO_C$, where $C$ is supported on the fiber of the map $p$. Therefore, the possible classes for $C$ are only $c_1(\cO_C)=h-e$ or $2h-2e$, but both classes violate the stability of $F_2$. 

For $\alpha=\frac{4}{3}$, let $(s,F_1)$ be a stable pair with $c_1(F_1)=2h-2e$ and $\chi(F_1)=0$. Then the image of the section map $s:\cO_{\FF_1}\lr F_1$ is $\mathrm{im}(s)=\cO_C$, such that $c_1(\cO_C)=h-e$ or $2h-2e$, which is a contradiction to the stability of $(s,F_1)$.

Finally, for $\alpha=13$, let $(s,F_1)$ be the stable pair with $c_1(F_1)=3h-e$ and $\chi(F_1)=-1$. The dual $F_1^D:=\cE xt^1(F_1, \omega_{\FF_1})$ of $F_1$ fits into the unique non-split extension 
\[
\ses{\cO_C}{F_1^D}{\CC_p}
\]
for $c_1(\cO_C)=3h-e$, $p\in C$, hence $F_1\cong F_1^{DD}=I_{p,C}$ (cf. \cite[Proposition 4.4]{CGKT18}). Since $h^0(F_1)=0$, the wall is empty.
\end{proof}

\begin{proposition}\label{wallcrossingofpairs}
There exist wall-crossings among moduli spaces $\bM_{\FF_1}^\alpha((4,2),2)$ of $\alpha$-stable pairs on $\FF_1$:
$$\bM_{\FF_1}^\infty((4,2),2)\dashleftarrow\dashrightarrow\bM_{\FF_1}^+((4,2),2),$$
where the blow-up centers at each wall are listed in Table~\ref{table3}.

\begin{table}
\begin{tabular}{|l|l|p{6.8cm}|}
\hline
$\alpha$&Blow-up center at $\alpha+\epsilon$&Blow-up center at $\alpha-\epsilon$\\
\hline
\hline
$\frac{4}{3}'$&a $\PP^3$-bundle over $\PP^2\times \PP^6$&a $\PP^2$-bundle over $\PP^2\times \PP^6$\\
\hline
$3$&a $\PP^2$-bundle over (a $\PP^1$-bundle over $\FF_1$)$\times\PP^7$&a $\PP^1$-bundle over (a $\PP^1$-bundle over $\FF_1$)$\times\PP^7$\\
\hline
$8$&a $\PP^3$-bundle over $\PP^8\times \PP^1$&a $\PP^1$-bundle over $\PP^8\times \PP^1$\\
\hline
$8'$&a $\PP^3$-bundle over $\PP^8$&a $\PP^2$-bundle over $\PP^8$\\
\hline
\end{tabular}
\medskip
\caption{Blow-up centers of $\bM_{\FF_1}((4,2),2)$}
\label{table3}
\end{table}

\end{proposition}

\begin{proof}
For wall $\alpha=\frac{4}{3}'$, let $(s,F_1)$ be the stable pair with $c_1(F_1)=3h-2e$ and $\chi(F_1)=1$. From \eqref{deggenus}, the section map $s:\cO_C\lr F_1$ must be an isomorphism and hence the pairs $(s,F_1)$ are parameterized by $|\cO_{\FF_1}(3h-2e)|\cong\PP^6$ (Proposition~\ref{secnum}). By the same argument, the locus for pairs $(0,F_2)$ with $c_1(F_2)=h$ and $\chi(F_2)=1$ is parameterized by $|\cO_{\FF_1}(h)|\cong\PP^2$. For these pairs, the wall crossing locus at $\alpha=\frac{4}{3}'+\epsilon$ parameterizes the non-split extensions
\[
\ses{(0,F_2)}{(s,F')}{(s,F_1)}.
\]
On the other hand, the wall locus at $\alpha=\frac{4}{3}'-\epsilon$ parameterizes the non-split extensions
\[
\ses{(s,F_1)}{(s,F'')}{(0,F_2)}.
\]
The results in Table~\ref{table3} follow since $\Ext^1((s,F_1),(0,F_2))=\CC^4$ and $\Ext^1((0,F_2),(s,F_1))=\CC^3$ (Proposition~\ref{secnum} and Proposition~\ref{defcoh}). 

The other cases can be derived by the same method and we omit the detail.
\end{proof}

We compare spaces $\bM_{\FF_1}^+((4,2),2)$ and $\bM_{\FF_1}((4,2),2)$. For the polystable sheaf $F \in \bM_{\FF_1}((4,2),2)\setminus \bM_{\FF_1}^s((4,2),2)$, $F \cong \cO_{C_{1}}\oplus \cO_{C_{2}}$ for some curves $C_1$ and $C_2$ with $c_1(\cO_{C_i})=2h-e$ for $i=1,2$. From Proposition~\ref{secnum}, the space $\mathrm{Sym}^2\PP^{4}$ parametrize such sheaves.

\begin{proposition}\label{prop:fiberofphi}      
Let $\phi : \bM_{\FF_1}^+((4,2),2) \to \bM_{\FF_1}((4,2),2)$ be the forgetful map $(s, F) \mapsto F$.
\begin{enumerate}
\item $\phi$ is a $\PP^1$-fibration over stable locus $\bM_{\FF_1}^s((4,2),2)$.
\item Let $\Delta \subset \mathrm{Sym}^2\PP^{4}$ be the diagonal. 
\begin{enumerate}
\item For $[F]\in \mathrm{Sym}^2\PP^{4} \setminus \Delta\subset \bM_{\FF^1}((4,2),2)\setminus \bM_{\FF^1}^s((4,2),2)$, the fiber $\phi^{-1}([F])=(s,F)$ parameterizes the non-split extension class
\[
\ses{(0,\cO_{C_2})}{(s,F)}{(s,\cO_{C_1})},
\]
which is a $(\PP^3-\{\mathrm{pt}\})$-bundle over $\PP^4\times \PP^4\setminus \PP^4$.
\item For $[F]\in \Delta$, the fiber $\phi^{-1}([F])$ parametrizes the unique pair $(s,\cO_{C_1}\oplus \cO_{C_2})$ such that $C_1\neq C_2$.
\end{enumerate}
\item Over $\Delta \cong \PP^{4} = |\cO_{\FF_1}(2h-e)| \subset\bM_{\FF^1}((4,2),2)\setminus \bM_{\FF^1}^s((4,2),2)$, $\phi$ is a $\PP^3$-fibration over its base space $\Delta$.
\end{enumerate}
\end{proposition}

\begin{proof}
Since $\chi(F)=2$ for each $F\in \bM_{\FF^1}((4,2),2)$, $h^0(F)\geq 2$. If $h^0(F)\geq 3$, then $h^0(F^D)\geq 1$ from the Serre duality. Hence there is a non-zero homomorphism $\cO_{C}\stackrel{s}{\lr} F^D$ for $\mathrm{Supp}(F^D)=C$, $c_1(\cO_C)=4h-2e$, which violates the semi-stability of $F^D$. Thus, $h^0(F)=2$ which implies item (1).

The remaining proof of the claim follows \cite[Proposition 3.6]{CM16} by changing the extension groups into
\[
\Ext_{\FF_1}^1(\cO_{C_1},\cO_{C_2} ) = \begin{cases}
        \CC^3, & \text{for } C_1\neq C_2 \in |\cO_{\FF_1}(2h-e)|\\
        \CC^4, & \text{for } C_1= C_2 \in |\cO_{\FF_1}(2h-e)|,
\end{cases}
\]
and we omit the detail.
\end{proof}

\begin{corollary}\label{cor2}
The virtual Poincar\'e polynomial of $\bM_{\mathbb{F}_1}((4,2), 2)$ is 
\[
        \rP(\bM_{\mathbb{F}_1}((4,2), 2))=1+3t^2+6t^4+8t^6+7t^8+7t^{10}+6t^{12}+8t^{14}+8t^{16}+10t^{18}+9t^{20}+8t^{22}+3t^{24}+t^{26}.
\]
\end{corollary}

\begin{proof}
From Proposition~\ref{wallcrossingofpairs}, 
\begin{align*}
\rP(\bM_{\FF_1}^+((4,2),2))&=\rP(\bM_{\FF^1}^\infty((4,2),2))+(\rP(\PP^2)-\rP(\PP^3))E(\PP^8)+(\rP(\PP^1)-\rP(\PP^3))\rP(\PP^8)\rP(\PP^1)\\
&+\rP(\FF_1)\rP(\PP^7)\rP(\PP^1)(\rP(\PP^1)-\rP(\PP^2))+\rP(\PP^6)\rP(\PP^2)(\rP(\PP^2)-\rP(\PP^3)),
\end{align*}
and from Proposition~\ref{inftystable} $\rP(\mathrm{Hilb}^{3}(\FF_1))=t^{12}+3t^{10}+9t^{8}+14t^6+9t^4+3t^2+1$ (\cite{GS93}) and $\rP(\bM_{\FF_1}^\infty((4,2),2))=\rP(\mathrm{Hilb}^{3}(\FF_1))\cdot \rP(\PP^8)$. Thus, 
\[\begin{split}
        \rP(\bM_{\FF_1}^+((4,2),2))&=
        t^{28}+4t^{26}+11t^{24}+18t^{22}+23t^{20}+24t^{18}+24t^{16}\\
        &+24t^{14}+24t^{12}+24t^{10}+23t^8+18t^6+11t^4+4t^2+1.
\end{split}\]
On the other hand, from Propositions~\ref{prop:fiberofphi} and~\ref{proepoly}, 
\[\begin{split}
        \rP(\bM_{\FF_1}^+((4,2),2))=& \rP(\PP^1)\rP(\bM_{\FF_1}^s((4,2),2))+ (\rP(\PP^3)-1)(\rP(\PP^4\times \PP^4)-\rP(\PP^4))+
        (\rP(\mathrm{Sym}^2\PP^4)\textendash \rP(\PP^4))\\
        &+\rP(\PP^3)\cdot \rP(\PP^4)
\end{split}\]
and thus we obtain $\rP(\bM_{\FF_1}^s((4,2),2))$. Finally, $\rP(\bM_{\FF_1}((4,2),2))=\rP(\bM_{\FF_1}^s((4,2),2))+\rP(\mathrm{Sym}^2\PP^4)$.
\end{proof}

\begin{corollary}\label{cor3}
 The virtual intersection Poincar\'e polynomial of $\bM_{\mathbb{F}_1}((4,2), 2)$ is 
 \[\begin{split}
\rIP(\bM_{\mathbb{F}_1}((4,2), 2))&= 1+3t^2+8t^4+10t^6+11t^8+11t^{10}+11t^{12}\\
&+11t^{14}+11t^{16}+11t^{18}+10t^{20}+8t^{22}+3t^{24}+t^{26}.
\end{split}\]
\end{corollary}

\begin{proof}
From Corollaries~\ref{comparisoneandie} and~\ref{cor2}, it is sufficient to calculate intersection cohomology for the open cones of singular locus in $\bM_{\mathbb{F}_1}((4,2), 2)$. Since the analytic neighborhoods of the strictly semi-stable loci of $\bM_4$ and $\bM_{\mathbb{F}_1}((4,2), 2)$ are isomorphic to each other (cf. Remark~\ref{luna} and Section~\ref{ihdoflocalsurface}, paragraph 1), the result follows from the result of Corollary~\ref{univrelation}.
\end{proof}
\begin{remark}
We use the intersection cohomology of the open cones of singular loci of the moduli spaces for the case $\mathbb{F}_1$ since we do not know any explicit birational relation among the relevant moduli spaces, unlike the cases $\mathbb{F}_0$ and $\PP^2$.
\end{remark}

\begin{remark}\label{imrem3}
The virtual intersection Poincar\'e polynomial of $\bM_{\mathbb{F}_1}((4,2), 2)$ in Proposition~\ref{cor3} is exactly the same as that of $\bM_{\mathbb{F}_1}((4,2),1)$ (\cite[Proposition 4.9]{CGKT18}).
\end{remark}
From Remarks~\ref{imrem1}, \ref{imrem2}, and~\ref{imrem3}, 
\begin{conjecture}
 The (virtual) intersection Poincar\'e polynomial of the space $\bM_S(c, \chi)$ depends only on the first Chern class $c$.
\end{conjecture}


\bibliographystyle{alpha} 
\newcommand{\etalchar}[1]{$^{#1}$}

\end{document}